\documentclass{article}

\usepackage{PRIMEarxiv}

\usepackage[utf8]{inputenc} 
\usepackage[T1]{fontenc}    
\usepackage{hyperref}       
\usepackage{url}            
\usepackage{booktabs}       
\usepackage{amsfonts}       
\usepackage{nicefrac}       
\usepackage{microtype}      
\usepackage{lipsum}
\usepackage{amssymb}
\usepackage{amsmath}
\usepackage{amsthm}
\usepackage{amsfonts}
\usepackage{mathtools}
\usepackage{fancyhdr}       
\usepackage{graphicx}       
\graphicspath{{media/}}     

\title{Exact random centering and hybrid fluctuation limits for interacting reinforced processes under critical Markov switching
}

\author{
  Hugo Cruz-Su\'arez \\
  Facultad de Ciencias F\'isico Matem\'aticas,  \\
  Benem\'erita Universidad Aut\'onoma de Puebla, \\
  Puebla, Pue., MÉXICO\\
  \texttt{hugo.cruz@correo.buap.mx} \\
}

\newcommand{\Pb}{\mathbb{P}}
\newcommand{\Eb}{\mathbb{E}}
\newcommand{\Rb}{\mathbb{R}}
\newcommand{\F}{\mathcal{F}}
\newtheorem{theorem}{Theorem}[section]
\newtheorem{corollary}[theorem]{Corollary}
\newtheorem{lemma}[theorem]{Lemma}
\newtheorem{proposition}[theorem]{Proposition}
\theoremstyle{definition}
\newtheorem{assumption}[theorem]{Assumption}
\theoremstyle{remark}
\newtheorem{remark}[theorem]{Remark}
\begin{document}
\maketitle

\begin{abstract}
We study two interacting reinforced occupation processes driven by a common two-state Markov environment whose switching probabilities decrease at the same rate as the stochastic-approximation gain. At this critical scale, the environmental motion persists in the first-order limit, which is a telegraph-driven piecewise deterministic Markov process. Consequently, centering at a deterministic equilibrium does not separate the order-one environmental response from the intrinsic square-root fluctuations. To overcome this obstruction, we introduce an exact discrete random centering that removes the accumulated environmental forcing before rescaling. An exact decomposition into collective and synchronization coordinates then separates the two contraction mechanisms of the system. Under the contraction condition corresponding to the square-root regime, we prove joint functional convergence of the randomly centered coordinates and the first-order background to a stationary hybrid PDMP--diffusion system. Its stationary entrance law is characterized through stochastic convolutions over the remote logarithmic past. Conditionally on the complete environmental path, the two fluctuation modes are independent centered Gaussian processes with a common state-dependent volatility and different contraction rates. We also determine the invariant distribution of the first-order component, derive explicit stationary second moments, and obtain a closed-form correlation formula that recovers interaction information not visible in the synchronized first-order limit.

\end{abstract}


\section{Introduction}
\label{seccion_introduction}

Fluctuation limits for stochastic approximation are commonly obtained by centering the recursion at a deterministic equilibrium or at an almost surely convergent random target. Once this first-order behavior has been removed, the remaining martingale noise can often be analyzed at the square-root scale. This approach becomes inadequate when the driving environment evolves on the same asymptotic scale as the stochastic-approximation gain. In that case, the environmental response may survive in the first-order limit as a nondegenerate stochastic process. Deterministic centering then leaves an order-one component that dominates the intrinsic fluctuations, while direct centering at the weak continuous-time limit does not provide the approximation accuracy required at the square-root scale.

The purpose of this paper is to resolve this centering problem for a pair of interacting reinforced occupation processes subject to critical Markov switching. The central device is an exact discrete random centering that reproduces the accumulated environmental response of the original recursion. Subtracting this response before rescaling separates the persistent environmental dynamics from the intrinsic martingale fluctuations and provides the basis for the second-order analysis developed below.

Reinforcement and environmental persistence are distinct mechanisms through which stochastic systems retain information from the past. Reinforced random walks and P\'olya urns are classical models in which past observations influence future sampling probabilities \cite{Pemantle_90,Pemantle_07}. Their asymptotic analysis is closely connected with stochastic approximation, martingale methods, and techniques from dynamical systems \cite{Benaim_1999,Pemantle_07}.

Interacting reinforced processes provide a natural framework for repeated binary decisions in which individual behavior is shaped both by past experience and by interaction with other components \cite{Crimaldi_IR_19,Aletti_G_2020}. Such mechanisms may be interpreted, for instance, in models of collective choice, coordination, or opinion formation \cite{Crimaldi_2023}, where repeated decisions reinforce previous tendencies while social interaction may promote alignment across individuals. When all components are additionally exposed to a common changing environment, however, similarity between their occupation proportions may arise not only from interaction but also from shared external forcing. Distinguishing these two mechanisms is therefore important, since apparent agreement at the first-order level need not reveal whether it is primarily associated with cross-interaction or with a persistent common environment. This motivates the study of whether the corresponding fluctuation structure can retain information that is not visible from the synchronized first-order behavior.

The critical phenomenon considered here arises from the balance between the stochastic-approximation gain $r_n\sim n^{-1}$ and the switching probability $q_n\sim c n^{-1}$ of a common Markov environment. At this critical switching scale, the environmental response remains dynamic in the first-order limit rather than averaging out or freezing. Consequently, deterministic centering does not remove the order-one environmental contribution, which dominates the intrinsic square-root fluctuations. This leads naturally to the exact random-centering procedure developed below and, ultimately, to a hybrid PDMP--diffusion limit rather than an autonomous Gaussian fluctuation process.

These questions are closely related to the literature on synchronization and fluctuations in interacting reinforced systems. Mean-field synchronization was studied in \cite{Dai_Pra_14}, while fluctuations around the synchronized regime were analyzed in \cite{Crimaldi_Irene_16}. Network-based interaction models and their asymptotic behavior were developed in \cite{Aletti_Gia_17,Aletti_19}, whereas functional fluctuation limits for interacting reinforced random walks were established in \cite{Crimaldi_IR_19}. Related second-order analyses have also been used to develop statistical procedures for estimating interaction structures and strengths \cite{Aletti_G_2020,Aleti_G_2026}. Subsequent work has addressed nonlinear reinforcement and possible failure of synchronization \cite{Crimaldi_2023}, complete first-order classifications on general networks \cite{Aletti_GIO_2024}, mechanisms combining individual and collective reinforcement \cite{Mirebrahimi_Mae_24}, phase transitions in nonlinear or power-reinforced interacting urns \cite{Costa_Jordan_22,Xia_26}, and competitive or cooperative interaction structures \cite{Aleandri_DaiPra_26}. The question addressed here is how this fluctuation structure is modified when the interacting components are driven by a common environment whose switching rate decreases over time.

A complementary line of research concerns stochastic approximation driven by time-inhomogeneous Markov noise. Decreasing-step tracking under Markov noise with slowly varying statistics is studied in \cite{Tsitsiklis_03}. For constant-step schemes with regime switching, continuous-time interpolations may converge to switching ordinary differential equations, while rescaled tracking errors converge to switching diffusions \cite{YinKris_04}. Related limits for least-mean-square algorithms are obtained in \cite{YinKris_05}. Central limit theorems in this setting are typically formulated after convergence to a stable target and under ergodicity assumptions that permit a Poisson-equation decomposition of the Markov noise \cite{Fort_15}.

When the off-diagonal transition probabilities of a finite-state Markov chain decay as $n^{-\gamma}$, its occupation behavior exhibits averaging for $0<\gamma<1$, a nondegenerate critical regime for $\gamma=1$, and freezing for $\gamma>1$ \cite{Dietz_Seth_07,Bouguet_Cloez_18}. Related changes of scale occur in the coin-turning model \cite{Englan_Volk_18}; at the critical order $n^{-1}$, the associated walk has a nonclassical continuous, piecewise-linear limit with Beta-distributed marginals \cite{Englan_Volk_20}. More generally, the theory of asymptotic pseudotrajectories connects time-inhomogeneous Markov chains with homogeneous continuous-time Markov processes \cite{Benaim_Bouguet_17}. For freezing empirical measures, this leads to exponential zig-zag piecewise deterministic Markov processes and functional convergence \cite{Bouguet_Cloez_18}.

The setting considered here lies at the intersection of these two lines of research. The spectral gap of the environmental transition kernel vanishes at the same order as the stochastic-approximation gain. Hence, the environment neither averages out nor freezes on the logarithmic time scale; instead, it generates a nondegenerate stochastic first-order limit. This is precisely why the conventional fluctuation argument breaks down. Deterministic centering retains an order-one environmental component, and direct substitution of the weak PDMP limit is not sufficient because weak convergence alone does not control the centering error at the square-root scale.

We consider two reinforced occupation processes on $\{0,1\}$ evolving in a common two-state Markov environment. For $i\in\{1,2\}$, let $X_n^i$ denote the occupation proportion of state $1$ in component $i$, and let $E_n\in\{-1,+1\}$ denote the environmental sign. The conditional update of each component depends on its own occupation proportion with coefficient $\rho$, on that of the other component with coefficient $\kappa$, and on the common forcing $\delta E_n$. The environment changes sign with probability $q_n$, where $n^\gamma q_n\longrightarrow c>0$. Thus, the reciprocal switching scale $q_n^{-1}$ is of order $n^\gamma$, whereas the occupation proportions evolve with a gain of order $1/n$.

The coordinates
$$
U_n=\frac{X_n^1+X_n^2}{2}-\frac12,\qquad
V_n=\frac{X_n^1-X_n^2}{2}
$$
yield an exact modal decomposition. The collective coordinate $U_n$ contains the environmental forcing and has contraction parameter $a=1-\rho-\kappa$, whereas the synchronization coordinate $V_n$ contains no direct environmental forcing and has contraction parameter $b=1-\rho+\kappa$. In particular, $V_n\to0$ in every switching regime, so the first-order transition is determined by the collective mode.

As a first-order classification, we establish an averaging--critical--freezing trichotomy. In the averaging regime, both occupation proportions converge to $1/2$. In the freezing regime, the environment eventually becomes constant and selects the common limiting level. At the critical scale, the logarithmic-time interpolation of $(U_n,E_n)$ converges to a stationary PDMP $(\mathcal U,\mathcal E)$, where $\mathcal E$ is a symmetric telegraph process and
$$
\frac{d\mathcal U(t)}{dt}
=
-a\mathcal U(t)+\delta\mathcal E(t)
$$
between successive environmental switches. After an affine transformation and a deterministic time change, this process belongs to the exponential zig-zag framework of \cite{Bouguet_Cloez_18} and is closely related to the critical coin-turning limit of \cite{Englan_Volk_20}. We provide a self-contained derivation and determine its invariant distribution. In particular, the invariant marginal distribution of $\mathcal U$ is an affine transformation of a symmetric Beta distribution. This classification supplies the first-order background required for the subsequent fluctuation analysis.

The main result concerns the residual sampling fluctuations in the critical switching regime. We introduce the exact discrete environmental response $\bar U_n$ and the deterministic synchronization response $\bar V_n$. The residuals $Y_n^U=U_n-\bar U_n$ and $Y_n^V=V_n-\bar V_n$ are driven solely by the centered Bernoulli innovations. Unlike centering at a deterministic equilibrium or directly at the limiting PDMP, this construction is exact at the discrete level and isolates the fluctuations relevant to the second-order analysis.

Under the contraction condition $a>1/2$, and hence also $b>1/2$, the square-root-normalized residual coordinates converge jointly with the first-order background. The limiting system consists of $(\mathcal U,\mathcal E)$ and two Ornstein--Uhlenbeck-type modes satisfying
$$
\begin{aligned}
d\mathcal Z^U(t)
&=
-\left(a-\frac12\right)\mathcal Z^U(t)\,dt
+\sigma\bigl(\mathcal U(t),\mathcal E(t)\bigr)\,dW^U(t),\\
d\mathcal Z^V(t)
&=
-\left(b-\frac12\right)\mathcal Z^V(t)\,dt
+\sigma\bigl(\mathcal U(t),\mathcal E(t)\bigr)\,dW^V(t).
\end{aligned}
$$
Here $W^U$ and $W^V$ are independent two-sided Brownian motions, independent of the telegraph process. The full limiting system is constructed from the remote logarithmic past through two-sided stochastic convolutions, which identify the stationary entrance law inherited from the complete prelimit history.

Conditionally on the complete environmental path, the two fluctuation modes are independent centered Gaussian processes with a common state-dependent volatility but different contraction rates. We derive their stationary second moments and show that the stationary contemporaneous correlation between the centered fluctuations of the original components is
$$
\chi=\frac{b-a}{a+b-1}
=\frac{2\kappa}{1-2\rho}.
$$
The first-order synchronized dynamics depend on $(\rho,\kappa)$ only through $a$, whereas the correlation also contains the difference between the contraction rates. Consequently,
$$
\kappa=\frac{\chi(2a-1)}{2(1-\chi)}
\qquad\text{and}\qquad
\rho=1-a-\kappa.
$$
Thus, at the level of the limiting covariance structure, the pair $(a,\chi)$ provides an asymptotic identifiability relation for the self-interaction and cross-interaction coefficients within the admissible parameter region.

The analytical contribution is threefold. First, the exact discrete environmental response provides a random centering that separates persistent environmental motion from intrinsic innovations without requiring deterministic first-order convergence. Second, the joint functional limit distinguishes the collective and synchronization modes within a stationary hybrid PDMP--diffusion system whose stationary entrance law is represented by two-sided stochastic convolutions. Third, the explicit covariance structure shows that second-order fluctuations recover interaction information that is not visible in the synchronized first-order limit.

The paper is organized as follows. Section \ref{seccion_model} introduces the stochastic recursions and their collective--synchronization decomposition, states the main results, and collects the preliminary estimates. Section \ref{seccion_first-order-regimes} establishes the first-order limits under the competition between the stochastic-approximation and environmental time scales: Subsections \ref{seccion_noncritical} and \ref{seccion_critical} treat the averaging and freezing regimes and the critical switching regime, respectively. Section \ref{seccion_second-order} develops the exact random centering, proves the hybrid functional central limit theorem, and derives the stationary covariance formulas. Section \ref{seccion_conclusion} presents the concluding remarks. Appendix \ref{appendice_invariant-law} contains the proof of the invariant-law formula.

Throughout, for a Polish state space $\mathsf E$, the Skorokhod space $D([s,t],\mathsf E)$ is endowed with the $J_1$ topology. All equalities and inequalities involving conditional expectations or conditional probabilities are understood to hold $\mathbb P$-almost surely. The symbol $C$ denotes a finite positive constant, independent of the asymptotic indices, whose value may change from line to line; a subscript, as in $C_T$ or $C_f$, indicates possible dependence on the corresponding fixed quantity.

\section{Stochastic recursions, modal decomposition, and main results}
\label{seccion_model}

\subsection{Reinforced stochastic recursions}

We now give the formal construction. For each $i\in\{1,2\}$, let $\{A_n^i:n\geq1\}$ be the sequence of states visited by process $i$, with $A_n^i\in\{0,1\}$, and let $X_n^i$ denote its reinforced occupation proportion of state $1$. Given $\nu>0$ and deterministic initial values $X_0^1,X_0^2\in[0,1]$, write $r_n=(n+\nu+1)^{-1}$. For every $n\geq0$, set
\begin{equation}\label{ecua_occupation-recursion}
X_{n+1}^i=X_n^i+r_n\left(A_{n+1}^i-X_n^i\right),\qquad i=1,2.
\end{equation}
Equivalently,
$$X_n^i=\frac{\nu X_0^i+\sum_{k=1}^n A_k^i}{n+\nu}.$$

Let $E_0\in\{-1,+1\}$ have an arbitrary distribution, and let $\{B_n:n\geq1\}$ be independent Bernoulli variables, independent of $E_0$, with $\Pb(B_n=1)=q_n$. The environment is defined by $E_{n+1}=E_n(1-2B_{n+1})$. Thus, $B_{n+1}=1$ indicates a switch between the two environmental states. Let $\F_n=\sigma\left(X_0^1,X_0^2,E_0, A_k^1,A_k^2,B_k:1\leq k\leq n\right)$. Conditionally on $\F_n$, the variables $A_{n+1}^1$ and $A_{n+1}^2$ are independent Bernoulli variables, independent of $B_{n+1}$, with
\begin{align}
\Pb(A_{n+1}^1=1\mid\F_n)&=\frac12+\rho\left(X_n^1-\frac12\right)+\kappa\left(X_n^2-\frac12\right)+\delta E_n,\label{ecua_transition-walk1}\\
\Pb(A_{n+1}^2=1\mid\F_n)&=\frac12+\rho\left(X_n^2-\frac12\right)+\kappa\left(X_n^1-\frac12\right)+\delta E_n.\label{ecua_transition-walk2}
\end{align}
Here $\rho,\kappa\geq0$ and $\delta>0$. Throughout the paper, we assume
\begin{equation}\label{ecua_admissibility}
\rho+\kappa+2\delta<1.
\end{equation}
Indeed, since $X_n^i-1/2\in[-1/2,1/2]$, the absolute displacement of each conditional probability from $1/2$ is at most $(\rho+\kappa)/2+\delta<1/2$. Hence, the probabilities in \eqref{ecua_transition-walk1}--\eqref{ecua_transition-walk2} belong to $(0,1)$.

\begin{assumption}\label{sup_switching}
There exist constants $c>0$ and $\gamma>0$ such that $0<q_n\leq 1/2$ and $n^\gamma q_n\longrightarrow c$ as $n\to\infty$.
\end{assumption}

\subsection{Collective and synchronization modes}
Recall the coordinates announced in the Introduction:
\begin{equation}\label{ecua_UV-definition}
U_n=\frac{X_n^1+X_n^2}{2}-\frac12,\qquad V_n=\frac{X_n^1-X_n^2}{2}.
\end{equation}
Their inverse transformation is
\begin{equation}\label{ecua_inverse-UV}
X_n^1=\frac12+U_n+V_n,\qquad X_n^2=\frac12+U_n-V_n.
\end{equation}

Set
\begin{equation}\label{ecua_a-b-definition}
a=1-\rho-\kappa\qquad\text{and}\qquad b=1-\rho+\kappa.
\end{equation}
Condition \eqref{ecua_admissibility} implies that $0<a\leq1$, $a>2\delta$, and $b=a+2\kappa>0$. In particular, $0<\delta/a<1/2$, and hence the two environmentally selected occupation levels $1/2-\delta/a$ and $1/2+\delta/a$ belong to the interior of $(0,1)$.

Define
\begin{align*}
\Delta M_{n+1}^{U}&=\frac{A_{n+1}^1+A_{n+1}^2}{2}-\Eb\left(\left.\frac{A_{n+1}^1+A_{n+1}^2}{2}\right|\F_n\right),\\
\Delta M_{n+1}^{V}&=\frac{A_{n+1}^1-A_{n+1}^2}{2}-\Eb\left(\left.\frac{A_{n+1}^1-A_{n+1}^2}{2}\right|\F_n\right).
\end{align*}
These variables satisfy $$\Eb\left[\Delta M_{n+1}^{U}\mid\F_n\right]=\Eb\left[\Delta M_{n+1}^{V}\mid\F_n\right]=0,$$ and $|\Delta M_{n+1}^{U}|\leq1$, $|\Delta M_{n+1}^{V}|\leq1$.

The following proposition gives the fundamental decomposition of the system.
\begin{proposition} \label{propo_UV-recursions}
For every $n\geq0$, the collective and synchronization coordinates satisfy
\begin{align}
U_{n+1}&=U_n+r_n\left(-aU_n+\delta E_n+\Delta M_{n+1}^{U}\right),\label{ecua_U-recursion}\\
V_{n+1}&=V_n+r_n\left(-bV_n+\Delta M_{n+1}^{V}\right).\label{ecua_V-recursion}
\end{align}
\end{proposition}

\begin{proof}
Adding and subtracting \eqref{ecua_transition-walk1} and \eqref{ecua_transition-walk2} gives
\begin{align*}
\Eb\left(\left.\frac{A_{n+1}^1+A_{n+1}^2}{2}\right|\F_n\right)&=\frac12+(\rho+\kappa)U_n+\delta E_n,\\
\Eb\left(\left.\frac{A_{n+1}^1-A_{n+1}^2}{2}\right|\F_n\right)&=(\rho-\kappa)V_n.
\end{align*}
Substitution into \eqref{ecua_occupation-recursion}, together with \eqref{ecua_UV-definition} and \eqref{ecua_a-b-definition}, yields \eqref{ecua_U-recursion}--\eqref{ecua_V-recursion}.
\end{proof}

As anticipated in the Introduction, \eqref{ecua_U-recursion}--\eqref{ecua_V-recursion} separate the collective and synchronization mechanisms. The recursion for $V_n$ is autonomous, whereas $U_n$ contains the forcing term $\delta E_n$; hence the environmental phase transition is determined by \eqref{ecua_U-recursion}.
\subsection{First-order asymptotics}
For $t\geq0$, define
\begin{equation} \label{ecua_log-interpolations}
\mathcal U_n(t)=U_{\lfloor ne^t\rfloor},\qquad \mathcal E_n(t)=E_{\lfloor ne^t\rfloor},\qquad \text{and} \qquad\mathcal X_n^i(t)=X_{\lfloor ne^t\rfloor}^i,\quad i=1,2.
\end{equation}
The state space of the limiting critical process is
$$\mathsf{S}=\left[-\frac{\delta}{a},\frac{\delta}{a}\right]\times\{-1,+1\}.$$

\begin{theorem}\label{teo_main}
Suppose that \eqref{ecua_admissibility} and Assumption \ref{sup_switching} hold.
\begin{enumerate}
\item [(i)] If $0<\gamma<1$, then
$$(X_n^1,X_n^2)\longrightarrow\left(\frac12,\frac12\right)\qquad\text{almost surely and in }L^2,$$
where the $L^2$ statement means that
$$\Eb\left[\left|X_n^1-\frac12\right|^2+\left|X_n^2-\frac12\right|^2\right]\longrightarrow0.$$
\item [(ii)] If $\gamma=1$, then for every $T>0$, the logarithmic-time
interpolations defined in \eqref{ecua_log-interpolations} satisfy
\begin{equation}\label{ecua_main-critical-limit}
(\mathcal U_n,\mathcal E_n)\Longrightarrow (\mathcal U,\mathcal E)
\end{equation}
in $D([0,T],\Rb\times\{-1,+1\})$, where $(\mathcal U,\mathcal E)$ is the stationary PDMP on $\mathsf S$ whose generator acts on the restrictions to $\mathsf S$ of functions $\varphi$ satisfying $\varphi(\cdot,e)\in C_b^1(\Rb)$ for each
$e\in\{-1,+1\}$ as
\begin{equation}\label{ecua_main-generator}
\mathcal L\varphi(u,e)=(-au+\delta e)\partial_u\varphi(u,e)+c\,[\varphi(u,-e)-\varphi(u,e)].
\end{equation}
The stationary initialization is constructed in Proposition \ref{propo_critical-marginal-limit}, and its invariant distribution is characterized explicitly in Proposition \ref{propo_critical-invariant-distribution}. Moreover,
\begin{equation}\label{ecua_main-critical-occupation-limit}
(\mathcal X_n^1,\mathcal X_n^2,\mathcal E_n)\Longrightarrow\left(\frac12+\mathcal U,\frac12+\mathcal U,\mathcal E\right)
\end{equation}
in $D([0,T],\Rb^2\times\{-1,+1\})$. In particular,
$$(X_n^1,X_n^2)\Longrightarrow \left(\frac12+U^{(c)},\frac12+U^{(c)}\right),$$
where $U^{(c)}$ has the invariant $\mathcal U$-marginal of the limiting process, characterized explicitly in Proposition \ref{propo_critical-invariant-distribution}.
\item [(iii)] If $\gamma>1$, there exists $E_\infty\in\{-1,+1\}$ such that $E_n=E_\infty$ eventually almost surely and
$$(X_n^1,X_n^2)\longrightarrow\left(\frac12+\frac{\delta}{a}E_\infty,\frac12+\frac{\delta}{a}E_\infty\right)\qquad\text{almost surely}.$$
\end{enumerate}
\end{theorem}

\subsection{Second-order main result}

The critical first-order limit is random, so centering at a deterministic equilibrium would mix environmental motion with intrinsic sampling noise. Instead, define the exact discrete responses $\bar U_0=U_0$, $\bar U_{n+1}=\bar U_n+r_n(-a\bar U_n+\delta E_n)$, $\bar V_0=V_0$, and $\bar V_{n+1}=(1-br_n)\bar V_n$. For $m_n(t)=\lfloor ne^t\rfloor$, put
\begin{align*}
\bar{\mathcal U}_n(t)=\bar U_{m_n(t)},\qquad \mathcal Z_n^U(t)=\sqrt{m_n(t)}\,[U_{m_n(t)}-\bar U_{m_n(t)}], \qquad\text{and}\qquad \mathcal Z_n^V(t)=\sqrt{m_n(t)}\,[V_{m_n(t)}-\bar V_{m_n(t)}].
\end{align*}

Finally, for $(u,e)\in[-1/2,1/2]\times\{-1,+1\}$, define
\begin{equation}\label{ecua_variance-function}
p(u,e)=\frac12+(1-a)u+\delta e
\qquad\text{and}\qquad
\sigma^2(u,e)=\frac12p(u,e)[1-p(u,e)].
\end{equation}
By \eqref{ecua_admissibility},
$$|(1-a)u+\delta e|\leq \frac{1-a}{2}+\delta=\frac{\rho+\kappa}{2}+\delta<\frac12.$$
Consequently, $p(u,e)\in(0,1)$ throughout the feasible coordinate range. Since $b=a+2\kappa$, the condition
$a>1/2$ also implies $b>1/2$.

\begin{theorem}\label{teo_hybrid-fluctuation-limit}
Suppose that \eqref{ecua_admissibility} holds and that Assumption \ref{sup_switching} is satisfied with $\gamma=1$. If
$a>1/2$, then for every $T>0$, $\left(\bar{\mathcal U}_n,\mathcal E_n,\mathcal Z_n^U,\mathcal Z_n^V\right)\Longrightarrow\left(\mathcal U,\mathcal E,\mathcal Z^U,\mathcal Z^V\right)$ in $D\bigl([0,T],\Rb\times\{-1,+1\}\times\Rb^2\bigr).$

The limiting process is characterized as follows. The environmental coordinate $\mathcal E$ is a stationary symmetric telegraph process with switching rate $c$. Let $W^U$ and $W^V$ be independent two-sided Brownian motions, independent of $\mathcal E$. The continuous coordinates satisfy
\begin{align}
d\mathcal U(t)&=\bigl[-a\mathcal U(t)+\delta\mathcal E(t)\bigr]\,dt,\label{ecua_limit-SDE-U}\\
d\mathcal Z^U(t)&=-\left(a-\frac12\right)\mathcal Z^U(t)\,dt+\sigma\bigl(\mathcal U(t),\mathcal E(t)\bigr)\,dW^U(t),\label{ecua_limit-SDE-ZU}\\
d\mathcal Z^V(t)&=-\left(b-\frac12\right)\mathcal Z^V(t)\,dt+\sigma\bigl(\mathcal U(t),\mathcal E(t)\bigr)\,dW^V(t).\label{ecua_limit-SDE-ZV}
\end{align}

This switching system, defined by \eqref{ecua_limit-SDE-U}--\eqref{ecua_limit-SDE-ZV}, admits a stationary weak solution that is unique in law. Its stationary law is characterized by the two-sided representations \eqref{ecua_stationary-filter-representation}, \eqref{ecua_stationary-ZU-two-sided}, and \eqref{ecua_stationary-ZV-two-sided}. Conditional on the complete path of $\mathcal E$, the processes $\mathcal Z^U$ and $\mathcal Z^V$ are independent centered Gaussian processes. Without conditioning, however, they need not be independent, since their conditional variance profiles depend on the same environmental path.
\end{theorem}

\subsection{Preliminary estimates}
We first establish that the two reinforced processes synchronize regardless of the environmental switching regime.
\begin{proposition}\label{propo_synchronization}
Under condition \eqref{ecua_admissibility}, $V_n\longrightarrow0$ almost surely and in $L^2$. Consequently, $X_n^1-X_n^2\longrightarrow0$ almost surely and in $L^2$.
\end{proposition}

\begin{proof}
From \eqref{ecua_V-recursion}, $V_{n+1}=(1-br_n)V_n+r_n\Delta M_{n+1}^{V}$. Taking conditional second moments and using the martingale-difference property gives
\begin{align*}
\Eb\left[V_{n+1}^2\mid\F_n\right]&=(1-br_n)^2V_n^2+r_n^2\Eb\left[\left(\Delta M_{n+1}^{V}\right)^2\middle|\F_n\right]\\
&\leq(1+b^2r_n^2)V_n^2-2br_nV_n^2+r_n^2.
\end{align*}
Since
$$\sum_{n=0}^{\infty}r_n=\infty,\qquad \sum_{n=0}^{\infty}r_n^2<\infty,$$
the Robbins--Siegmund theorem \cite{Robbins_71} implies that $V_n^2$ converges almost surely and that
$$\sum_{n=0}^{\infty}r_nV_n^2<\infty \qquad\text{almost surely}.$$
Let $L=\lim_nV_n^2$. If $L>0$, then $V_n^2\geq L/2$ eventually, which, together with $\sum_nr_n=\infty$, would imply
$\sum_nr_nV_n^2=\infty$. Hence, $L=0$ almost surely, and therefore $V_n\longrightarrow0$ almost surely. Since $|V_n|\leq1/2$, bounded convergence also gives $V_n\to0$ in $L^2$. The conclusions for the original coordinates follow from $X_n^1-X_n^2=2V_n$.
\end{proof}

Proposition \ref{propo_synchronization} holds for every $\gamma>0$. Thus, in all three environmental regimes, the two occupation proportions become asymptotically indistinguishable. It remains to determine the behavior of their common collective component.

For integers $0\leq k\leq n$, define
\begin{equation}\label{ecua_product-Pkn}
P_{k,n}=\prod_{j=k}^{n-1}\left(1-\frac{a}{j+\nu+1}\right),\qquad\text{and set}\qquad P_{n,n}=1.
\end{equation}
Iterating \eqref{ecua_U-recursion}, with the products defined in \eqref{ecua_product-Pkn}, gives
\begin{equation}\label{ecua_U-variation-constants}
U_n=P_{0,n}U_0+\delta\sum_{k=0}^{n-1}\frac{P_{k+1,n}}{k+\nu+1}E_k+\sum_{k=0}^{n-1}\frac{P_{k+1,n}}{k+\nu+1}\Delta M_{k+1}^{U}.
\end{equation}

The three terms in \eqref{ecua_U-variation-constants} represent, respectively, the contribution of the initial condition, the accumulated environmental forcing, and the martingale fluctuations.

The product $P_{k,n}$ has the explicit representation
$$P_{k,n}=\frac{\Gamma(n+\nu+1-a)\,\Gamma(k+\nu+1)}{\Gamma(n+\nu+1)\,\Gamma(k+\nu+1-a)}.$$
Consequently, there exists a constant $C>0$, independent of $k$ and $n$, such that
\begin{equation}\label{ecua_Pkn-upper-bound}
P_{k,n}\leq C\left(\frac{k+\nu+1}{n+\nu+1}\right)^a,\qquad 0\leq k\leq n.
\end{equation}
Moreover, for every sequence $k_n$ such that $k_n\to\infty$ and $k_n\leq n$,
$$P_{k_n,n}\sim\left(\frac{k_n+\nu+1}{n+\nu+1}\right)^a.$$
In particular, $P_{0,n}U_0\longrightarrow0$.

Define the environmental and martingale components by
\begin{align}\label{ecua_Rn-definition}
H_n&=\delta\sum_{k=0}^{n-1}\frac{P_{k+1,n}}{k+\nu+1}E_k,\notag\\
R_n&=\sum_{k=0}^{n-1}\frac{P_{k+1,n}}{k+\nu+1}\Delta M_{k+1}^{U}.
\end{align}
The identity $P_{k+1,n}-P_{k,n}=a r_kP_{k+1,n}$ will be used repeatedly. In particular,
$$\sum_{k=j}^{n-1}r_kP_{k+1,n}=\frac{1-P_{j,n}}{a},\qquad|H_n|\leq\frac{\delta}{a}.$$
Then
\begin{equation}\label{ecua_U-HR-decomposition}
U_n=P_{0,n}U_0+H_n+R_n.
\end{equation}

The martingale component is asymptotically negligible at the scale relevant for the first-order phase transition.

\begin{lemma}\label{lema_martingale-negligible}
The sequence $R_n$ satisfies $R_n\longrightarrow0$ almost surely and in $L^2$. In fact,
$$\Eb[R_n^2]=\begin{cases}
O(n^{-2a}), & 0<a<\frac12,\\[1mm]
O(n^{-1}\log n), & a=\frac12,\\[1mm]
O(n^{-1}), & a>\frac12.
\end{cases}$$
\end{lemma}

\begin{proof}
By the orthogonality of martingale differences and the uniform bound \eqref{ecua_Pkn-upper-bound}, there exists $C>0$ such that
\begin{align*}
\Eb[R_n^2]&=\sum_{k=0}^{n-1}\frac{P_{k+1,n}^2}{(k+\nu+1)^2}\Eb\left[\left(\Delta M_{k+1}^{U}\right)^2\right]\\
&\leq C(n+\nu+1)^{-2a}\sum_{k=0}^{n-1}(k+\nu+1)^{2a-2}.
\end{align*}
The stated $L^2$ estimates follow by considering separately the cases $a<1/2$, $a=1/2$, and $a>1/2$. In particular, $R_n\longrightarrow0$ in $L^2$.

To prove almost-sure convergence, observe from \eqref{ecua_Rn-definition} that
$$R_{n+1}=\left(1-\frac{a}{n+\nu+1}\right)R_n+\frac{1}{n+\nu+1}\Delta M_{n+1}^{U}.$$ Therefore,
$$\Eb\left[R_{n+1}^2\mid\F_n\right]\leq\left(1+a^2r_n^2\right)R_n^2-2ar_nR_n^2+r_n^2.$$ Another application of the Robbins--Siegmund theorem \cite{Robbins_71} shows that $R_n^2$ converges almost surely and that
$$\sum_{n=0}^{\infty}r_nR_n^2<\infty\qquad\text{almost surely}.$$
Let $L_R=\lim_nR_n^2$. If $L_R>0$, then $R_n^2\geq L_R/2$ eventually, which would contradict $\sum_nr_nR_n^2<\infty$ because $\sum_nr_n=\infty$. Hence, $L_R=0$ almost surely, and therefore $R_n\longrightarrow0$ almost surely.
\end{proof}

Combining \eqref{ecua_U-HR-decomposition} with Lemma \ref{lema_martingale-negligible}, we obtain $U_n-H_n\longrightarrow0$ almost surely and in $L^2$. Thus, the first-order asymptotic behavior of the collective coordinate is determined entirely by the weighted environmental term $H_n$. The following sections analyze this term under the three possible relations between the environmental switching probability $q_n$ and the stochastic-approximation step size $1/n$.
\section{First-order limits under time-scale competition} \label{seccion_first-order-regimes}
\subsection{Averaging and freezing regimes} \label{seccion_noncritical}
We begin with the averaging and freezing regimes of Theorem \ref{teo_main}. Their proofs also clarify why the critical case $\gamma=1$ requires a separate functional analysis.

Suppose first that $0<\gamma<1$. For $0\leq k\leq \ell$, independence of the switch indicators gives
$$\Eb[E_kE_\ell]=\prod_{j=k+1}^{\ell}(1-2q_j).$$ Assumption \ref{sup_switching} implies the uniform
correlation-sum bound
$$\sup_{k\geq0}\frac{1}{(k+1)^\gamma}\sum_{\ell=0}^{\infty}|\Eb[E_kE_\ell]|<\infty.$$
Indeed, $q_j\geq c_0j^{-\gamma}$ for all sufficiently large $j$, and comparison with the integral of $\exp\{-c_1[(k+x)^{1-\gamma}-k^{1-\gamma}]\}$ gives an upper bound of order $(k+1)^\gamma$ in the forward direction. The same argument applies backward from $k$, and the finitely many remaining indices are absorbed into the constant.

Put $w_{k,n}=r_kP_{k+1,n}$. Since these weights are nonnegative, the preceding correlation bound, the inequality $2xy\leq x^2+y^2$, and \eqref{ecua_Pkn-upper-bound} yield
$$\Eb[H_n^2]\leq C\sum_{k=0}^{n-1}w_{k,n}^2(k+1)^\gamma \leq C(n+\nu+1)^{-2a} \sum_{k=0}^{n-1}(k+\nu+1)^{2a+\gamma-2}.$$
Consequently,
$$\Eb[H_n^2]=\begin{cases}
O(n^{-2a}), & 2a<1-\gamma,\\
O(n^{\gamma-1}\log n), & 2a=1-\gamma,\\
O(n^{\gamma-1}), & 2a>1-\gamma,
\end{cases}$$ so $H_n\to0$ in $L^2$.

For almost-sure convergence, let $\eta=\min\{2a,1-\gamma\}>0$ and choose $\zeta>1/\eta$. With $n_j=\lfloor j^\zeta\rfloor$, the preceding bounds and Chebyshev's inequality imply that
$$\sum_j\Pb(|H_{n_j}|>\varepsilon)<\infty$$ for every $\varepsilon>0$. Hence, $H_{n_j}\to0$ almost surely. Moreover,
the recursion $H_{n+1}=(1-ar_n)H_n+\delta r_nE_n$ and the bound $|H_n|\leq\delta/a$ give
$$\max_{n_j\leq m<n_{j+1}}|H_m-H_{n_j}|\leq 2\delta\sum_{k=n_j}^{n_{j+1}-1}r_k\longrightarrow0.$$
Thus, $H_n\to0$ almost surely. Together with \eqref{ecua_U-HR-decomposition}, Lemma \ref{lema_martingale-negligible}, and Proposition \ref{propo_synchronization}, this proves part (i) of Theorem \ref{teo_main}.

Now suppose that $\gamma>1$. Then $\sum_nq_n<\infty$, so the first Borel--Cantelli lemma implies that only finitely many switches occur. Consequently, $E_n=E_\infty\in\{-1,+1\}$ eventually almost surely. If $m$ is after the last switch, the recursion for $H_n$ and the telescoping identity give
$$H_n=P_{m,n}H_m+\delta E_\infty\sum_{k=m}^{n-1}r_kP_{k+1,n}=P_{m,n}H_m+\frac{\delta}{a}E_\infty(1-P_{m,n}).$$
Since $P_{m,n}\to0$,
$$H_n\longrightarrow\frac{\delta}{a}E_\infty\qquad\text{almost surely}.$$
The bound $|H_n|\leq\delta/a$ also gives convergence in $L^2$. The same decomposition and synchronization arguments prove part (iii) of Theorem \ref{teo_main}.

\subsection{Critical switching regime}\label{seccion_critical}
Throughout this subsection, $\gamma=1$, so that
\begin{equation}\label{ecua_critical-switching-assumption}
nq_n\longrightarrow c>0.
\end{equation}
The critical first-order limit is a two-state exponential zig-zag process \cite{Bouguet_Cloez_18}, closely related to the critical coin-turning limit of \cite{Englan_Volk_20}. Its direct derivation in the present coordinates provides the joint stationary representation used in the second-order analysis.

\subsubsection{Logarithmic-time environment}

For fixed $T>0$, define $\widetilde{\mathcal E}_n(t)=E_{\lfloor ne^t\rfloor}$, $-T\leq t\leq T$.

\begin{proposition}\label{propo_log-environment}
For every $T>0$, $\widetilde{\mathcal E}_n\Longrightarrow\widetilde{\mathcal E}$ in $D([-T,T],\{-1,+1\})$, where $\widetilde{\mathcal E}$ is the restriction of a two-sided stationary symmetric telegraph process with switching rate $c$.
\end{proposition}

\begin{proof}
Let
$$m_n(t)=\lfloor ne^t\rfloor,\qquad J_n(t)=\sum_{j=m_n(-T)+1}^{m_n(t)}B_j,$$
and define its deterministic compensator by
$$\Lambda_n(t)=\sum_{j=m_n(-T)+1}^{m_n(t)}q_j,\qquad -T\leq t\leq T.$$
The relation $jq_j\to c$ and a logarithmic Riemann-sum comparison give
$$\sup_{-T\leq t\leq T}\left|\Lambda_n(t)-c(t+T)\right|\longrightarrow0,\qquad \max_{m_n(-T)<j\leq m_n(T)}q_j\longrightarrow0.$$
Le Cam's inequality \cite{Le_Cam_60} gives convergence of the counts on disjoint intervals to independent Poisson variables with the corresponding means.

For $t\in[-T,T]$, let $\mathcal F^E_{n,t}= \sigma(E_{m_n(-T)},B_j:m_n(-T)<j\le m_n(t))$. If $\tau_n\le T$ is an $\mathcal F^E_n$-stopping time and $0\le h_n\le h$, then
$$\mathbb P\!\left(J_n(\tau_n+h_n)-J_n(\tau_n)\ge1\right)\le\mathbb E\!\left[
\sum_{j=m_n(\tau_n)+1}^{m_n(\tau_n+h_n)}q_j\right]\le ch+o_n(1),$$
uniformly in $\tau_n$ and $h_n$, since $\sup_{m_n(-T)<j\le m_n(T)}|jq_j-c|\to0$ and the corresponding harmonic sum is bounded by $h+o_n(1)$. Moreover, $\sup_n\mathbb E[J_n(T)]<\infty$. Hence Aldous's criterion \cite{Aldous_1978} yields tightness of $(J_n)$ on $D([-T,T])$. Combined with the convergence of the finite-dimensional distributions, this proves convergence to a rate-$c$ Poisson process.

Moreover,
$$\Eb\left[E_m\right]=\Eb \left[E_0\prod_{j=1}^m(1-2q_j)\right]\longrightarrow0$$
because $\sum_jq_j=\infty$. Hence, $E_{m_n(-T)}$ converges to a symmetric random sign. For every $n$, this sign is independent of the subsequent switches $\{B_j:j>m_n(-T)\}$. Therefore, $\bigl(E_{m_n(-T)},J_n\bigr)\Longrightarrow\bigl(\widetilde{\mathcal E}(-T),J\bigr)$, where $\widetilde{\mathcal E}(-T)$ is a symmetric random sign independent of the rate-$c$ Poisson counting process $J$.

Since $\widetilde{\mathcal E}_n(t)=E_{m_n(-T)}(-1)^{J_n(t)}$, the continuous mapping theorem proves the result. The parity map is almost surely continuous at the limit because the Poisson process has only unit jumps and no fixed-time discontinuities. The symmetric initial law and the time-homogeneous switching rate make the limiting telegraph process stationary.
\end{proof}

Let $\widetilde{\mathcal E}$ be a two-sided stationary telegraph process with rate $c$, set $\mathcal E=\widetilde{\mathcal E}|_{[0,\infty)}$, and define
$$U^{(c)}=\delta\int_{-\infty}^{0}e^{as}\widetilde{\mathcal E}(s)\,ds.$$

\begin{proposition}\label{propo_critical-marginal-limit}
For every $T>0$, $(U_n,\mathcal E_n)\Longrightarrow(U^{(c)},\mathcal E)$ in $\Rb\times D([0,T],\{-1,+1\})$. In particular, $(U_n,E_n)\Longrightarrow \bigl(U^{(c)},\widetilde{\mathcal E}(0)\bigr)$.
\end{proposition}

\begin{proof}
For $S>0$, let $j_n(S)=\lfloor ne^{-S}\rfloor$ and define
$$H_n^{(S)}=\delta\sum_{k=j_n(S)}^{n-1}\frac{P_{k+1,n}}{k+\nu+1}E_k.$$
The telescoping identity and \eqref{ecua_Pkn-upper-bound} give
\begin{equation}\label{ecua_critical-tail-bound}
|H_n-H_n^{(S)}|\leq C(e^{-aS}+n^{-a}).
\end{equation}
For fixed $S$, the gamma-ratio estimate yields, uniformly for $j_n(S)\leq k<n$, $$\frac{P_{k+1,n}}{k+\nu+1}=n^{-1}(k/n)^{a-1}(1+o(1)).$$ Proposition \ref{propo_log-environment}, applied on $[-S,T]$, and the corresponding logarithmic Riemann-sum approximation imply
$$(H_n^{(S)},\mathcal E_n)\Longrightarrow\left(\delta\int_{-S}^{0}e^{as}\widetilde{\mathcal E}(s)\,ds,\mathcal E\right).$$
Letting $S\to\infty$ removes the truncation by \eqref{ecua_critical-tail-bound}. Finally, $U_n-H_n\longrightarrow0$ in probability by \eqref{ecua_U-HR-decomposition} and Lemma \ref{lema_martingale-negligible}. The converging-together theorem \cite[Theorem 3.2]{Billingsley_1999} then completes the proof.
\end{proof}

\subsubsection{Functional PDMP limit}
Define
$$\mathcal U(t)=e^{-at}U^{(c)}+\delta\int_0^te^{-a(t-s)}\mathcal E(s)\,ds.$$
Equivalently,
\begin{equation}\label{ecua_stationary-filter-representation}
\mathcal U(t)=\delta\int_{-\infty}^te^{-a(t-s)}\widetilde{\mathcal E}(s)\,ds.
\end{equation}

\begin{theorem}\label{teo_critical-functional-convergence}
For every $T>0$,
\begin{equation}\label{ecua_critical-functional-convergence}
(\mathcal U_n,\mathcal E_n)\Longrightarrow(\mathcal U,\mathcal E)
\end{equation}
in $D([0,T],\Rb\times\{-1,+1\})$. The limit is stationary. Its generator acts on the restrictions to $\mathsf S$ of functions $\varphi$ satisfying $\varphi(\cdot,e)\in C_b^1(\Rb)$ for each $e\in\{-1,+1\}$ as
\begin{equation}\label{ecua_PDMP-generator-proof}
\mathcal L\varphi(u,e)=(-au+\delta e)\partial_u\varphi(u,e)+c[\varphi(u,-e)-\varphi(u,e)].
\end{equation}
\end{theorem}

\begin{proof}
Let $F(u,e)=-au+\delta e$ and
$$M_n(t)=\sum_{k=n}^{\lfloor ne^t\rfloor-1}r_k\Delta M_{k+1}^{U}.$$
Since $r_k-\log((k+1)/k)=O(k^{-2})$, the recursion for $U_n$ gives
$$\mathcal U_n(t)=U_n+\int_0^t F(\mathcal U_n(s),\mathcal E_n(s))\,ds+M_n(t)+\varepsilon_n(t),$$
where
$$\sup_{t\leq T}|\varepsilon_n(t)|\leq \frac{C_T}{n}.$$
Doob's inequality gives
$$\Eb\sup_{t\leq T}|M_n(t)|^2\leq4\sum_{k=n}^{\lfloor ne^T\rfloor-1}r_k^2\leq\frac{C_T}{n}.$$

Define
$$\Phi(u,\varepsilon)(t)=e^{-at}u+\delta\int_0^t e^{-a(t-s)}\varepsilon(s)\,ds,$$
Gronwall's inequality therefore shows that
$$\sup_{t\leq T}\left|\mathcal U_n(t)-\Phi(U_n,\mathcal E_n)(t)\right|\longrightarrow0$$
in probability.

Proposition \ref{propo_critical-marginal-limit} yields the required joint convergence of the initial collective coordinate and the future environment. The map $\Phi$ is continuous at telegraph paths: $J_1$ convergence of uniformly bounded paths implies $L^1$ convergence whenever the limit has only finitely many jumps. The continuous mapping theorem proves \eqref{ecua_critical-functional-convergence}.

Representation \eqref{ecua_stationary-filter-representation} gives stationarity, while the deterministic flow between sign changes gives \eqref{ecua_PDMP-generator-proof}, in the standard PDMP framework of \cite{Davis_84}. This also proves \eqref{ecua_main-critical-limit} and \eqref{ecua_main-generator}.
\end{proof}

\begin{corollary} \label{corol_critical-occupation-functional-convergence}
For every $T>0$,
$$(\mathcal X_n^1,\mathcal X_n^2,\mathcal E_n)\Longrightarrow\left(\frac12+\mathcal U,\frac12+\mathcal U,\mathcal E\right)$$
in $D([0,T],\Rb^2\times\{-1,+1\})$.
\end{corollary}

\begin{proof}
Proposition \ref{propo_synchronization} implies
$$\sup_{0\leq t\leq T}|V_{\lfloor ne^t\rfloor}|\longrightarrow0\qquad\text{almost surely}.$$
The result follows from \eqref{ecua_inverse-UV} and Theorem \ref{teo_critical-functional-convergence}, and proves
\eqref{ecua_main-critical-occupation-limit}.
\end{proof}

\subsubsection{Invariant law}
The beta law below describes the continuous marginal of the invariant distribution of the two-state exponential zig-zag process arising in the freezing-chain setting \cite{Bouguet_Cloez_18}. Closely related stationary laws for linear relaxation driven by dichotomous or telegraph noise appear in \cite{Sancho_84,Font_2016,Ratanov_2022}. Its moments enter the second-order covariance formulas. The proof is deferred to Appendix \ref{appendice_invariant-law}.

\begin{proposition}
\label{propo_critical-invariant-distribution}
The PDMP in Theorem \ref{teo_critical-functional-convergence} has a unique invariant distribution. The marginal density of its continuous component is
\begin{equation}\label{ecua_critical-density-derived}
f_c(u)=\frac{a}{\delta}\frac{\Gamma(c/a+1/2)}{\sqrt{\pi}\,\Gamma(c/a)}\left[1-\left(\frac{au}{\delta}\right)^2\right]^{c/a-1},\qquad |u|<\frac{\delta}{a}.
\end{equation}
Equivalently,
\begin{equation}\label{ecua_critical-beta-derived}
\frac12\left(1+\frac{aU^{(c)}}{\delta}\right)\sim\operatorname{Beta}\left(\frac ca,\frac ca\right).
\end{equation}
The joint invariant densities, defined by $\Pb\left(U^{(c)}\in du,\widetilde{\mathcal E}(0)=e\right)=f_e(u)\,du$, $e\in\{-1,+1\}$, are
\begin{equation}\label{ecua_joint-densities}
f_+(u)=\frac12\left(1+\frac{au}{\delta}\right)f_c(u),\qquad f_-(u)=\frac12\left(1-\frac{au}{\delta}\right)f_c(u).
\end{equation}
\end{proposition}
In particular,
\begin{equation}\label{ecua_critical-moments}
\Eb[U^{(c)}]=0,\qquad \operatorname{Var}(U^{(c)})=\frac{\delta^2}{a(a+2c)},\qquad \Eb[U^{(c)}\widetilde{\mathcal E}(0)]=\frac{\delta}{a+2c}.
\end{equation}
The density is U-shaped if $c<a$, uniform if $c=a$, and unimodal if $c>a$. It has no mass at the endpoints.

\begin{corollary}\label{corol_no-probability-convergence}
In the critical regime, $U_n$ does not converge in probability.
\end{corollary}

\begin{proof}
Fix $t_0>0$. On the event that the limiting telegraph process has no switch on $[0,t_0]$, which has probability $e^{-ct_0}>0$,
$$\mathcal U(t_0)-\mathcal U(0)=(1-e^{-at_0})\left(\frac{\delta}{a}\mathcal E(0)-\mathcal U(0)\right).$$
This difference is not concentrated at zero because the invariant law has no mass at $\pm\delta/a$.

If $U_n$ converged in probability, then $U_{\lfloor ne^{t_0}\rfloor}-U_n\longrightarrow0$ in probability. On the other hand, Theorem \ref{teo_critical-functional-convergence} gives
$$U_{\lfloor ne^{t_0}\rfloor}-U_n\Longrightarrow\mathcal U(t_0)-\mathcal U(0),$$
whose distribution is not concentrated at zero. This contradiction completes the proof.
\end{proof}
\begin{remark}The following representation makes the connection with the coin-turning model explicit. Define the time-changed integrated
environmental process
$$\widehat S(t)=\int_0^t E_{\lfloor u^{1/a}\rfloor}\,du.$$
For every integer $m\geq1$,
$$\frac{\delta}{a}\frac{\widehat S(m^a)}{m^a}=\frac{\delta}{a m^a}\sum_{k=0}^{m-1}\bigl((k+1)^a-k^a\bigr)E_k.$$
Thus, this quantity is a normalized weighted integral of the coin-turning sequence $\{E_k: k\ge 0\}$ after the deterministic time change $u=k^a$.

The gamma-ratio estimate for $P_{k,m}$ and the mean-value theorem yield, for every $T>0$,
$$\sup_{n\leq m\leq\lfloor ne^T\rfloor}\sum_{k=0}^{m-1}\left|\frac{P_{k+1,m}}{k+\nu+1}-\frac{(k+1)^a-k^a}{a m^a}\right|\longrightarrow0.$$
Consequently, \eqref{ecua_U-HR-decomposition} and Lemma \ref{lema_martingale-negligible} imply that
$$\sup_{n\leq m\leq\lfloor ne^T\rfloor}\left|U_m-\frac{\delta}{a}\frac{\widehat S(m^a)}{m^a}\right|\longrightarrow0
\qquad\text{almost surely}.$$

Under the critical assumption $nq_n\to c$, the deterministic time change $u=k^a$ produces an asymptotic switching intensity $c/(au)$. Hence, $\widehat S(t)/t$ is naturally associated with the critical coin-turning scaling regime with effective parameter $c/a$; see
\cite{Englan_Volk_18,Englan_Volk_20}. This also explains the appearance of the ratio $c/a$ in the affine-Beta invariant law of the limiting PDMP.

This connection is used only to place the first-order limit within the coin-turning and freezing-chain literature. The second-order analysis requires the exact discrete environmental centering introduced in Section \ref{seccion_model}.
\end{remark}
\section{Exact random centering and hybrid fluctuation limits} \label{seccion_second-order}

The first-order limit retains the environmental forcing but removes the fresh Bernoulli innovations. We now center the discrete system by its exact environmental response and recover those innovations at the square-root scale. Throughout this section, \eqref{ecua_critical-switching-assumption} holds and $a>1/2$. Since $b=a+2\kappa$, this also implies $b>1/2$.
\subsection{Exact random centering and conditional covariance}
Recall the environmental response $(\bar U_n)$ and the deterministic synchronization response $(\bar V_n)$ from Section \ref{seccion_model}. Their recursions are $\bar U_0=U_0$, $\bar U_{n+1}=\bar U_n+r_n(-a\bar U_n+\delta E_n)$, $\bar V_0=V_0$, and $\bar V_{n+1}=(1-br_n)\bar V_n$. Thus, $\bar U_n=P_{0,n}U_0+H_n$ and $\bar V_n=O(n^{-b})$. Set $Y_n^U=U_n-\bar U_n$ and $Y_n^V=V_n-\bar V_n$. Subtracting these two recursions from \eqref{ecua_U-recursion}--\eqref{ecua_V-recursion} gives
\begin{align}
Y_{n+1}^U&=(1-ar_n)Y_n^U+r_n\Delta M_{n+1}^U,\label{ecua_centered-U-recursion}\\
Y_{n+1}^V&=(1-br_n)Y_n^V+r_n\Delta M_{n+1}^V.\label{ecua_centered-V-recursion}
\end{align}
In particular, $Y_n^U=R_n$ in the notation of \eqref{ecua_Rn-definition}.

For the conditional covariance calculation, write
$$p_n=\frac12+(1-a)U_n+\delta E_n,\qquad d_n=(\rho-\kappa)V_n.$$
The two Bernoulli parameters in
\eqref{ecua_transition-walk1}--\eqref{ecua_transition-walk2} are then
$p_n+d_n$ and $p_n-d_n$, respectively.

\begin{lemma}\label{lema_conditional-covariance}
For every $n\geq0$, $\Eb\left[(\Delta M_{n+1}^U)^2\mid\F_n\right]=\Eb\left[(\Delta M_{n+1}^V)^2\mid\F_n\right]=Q_n$ and $\Eb\left[\Delta M_{n+1}^U\Delta M_{n+1}^V\mid\F_n\right]=\Xi_n$, where
\begin{equation}\label{ecua_Qn-Xin}
Q_n=\frac12\left[p_n(1-p_n)-d_n^2\right],\qquad \Xi_n=\frac12d_n(1-2p_n).
\end{equation}
Consequently, if $U_n\to u$, $V_n\to0$, and $E_n=e$ along a sequence, then $Q_n\longrightarrow\sigma^2(u,e)$ and $\Xi_n\longrightarrow0$, where $p$ and $\sigma^2$ are defined in \eqref{ecua_variance-function}.
\end{lemma}

\begin{proof}
Conditionally on $\F_n$, set $\eta_{n+1}^i=A_{n+1}^i-\Eb[A_{n+1}^i\mid\F_n]$. Then
$$\Delta M_{n+1}^U=\frac{\eta_{n+1}^1+\eta_{n+1}^2}{2},\qquad \Delta M_{n+1}^V=\frac{\eta_{n+1}^1-\eta_{n+1}^2}{2}.$$
The variables $\eta_{n+1}^1$ and $\eta_{n+1}^2$ are conditionally independent. Applying $h(x)=x(1-x)$ to the parameters $p_n\pm d_n$ gives
$$\frac14\{h(p_n+d_n)+h(p_n-d_n)\}=\frac12\{p_n(1-p_n)-d_n^2\}$$
for both conditional variances, and
$$\frac14\{h(p_n+d_n)-h(p_n-d_n)\}=\frac12d_n(1-2p_n)$$
for the conditional covariance. The final assertions follow from $d_n\to0$ and continuity.
\end{proof}

\subsection{Hybrid functional central limit theorem}
For $t\geq0$, let $m_n(t)=\lfloor ne^t\rfloor$ and define $\bar{\mathcal U}_n(t)=\bar U_{m_n(t)}$, $\mathcal Z_n^U(t)=\sqrt{m_n(t)}\,Y_{m_n(t)}^U$, and $\mathcal Z_n^V(t)=\sqrt{m_n(t)}\,Y_{m_n(t)}^V$. Let $(\widetilde{\mathcal E}(t):t\in\Rb)$ be the two-sided stationary telegraph process introduced in Proposition \ref{propo_log-environment}. On an extension of the probability space, let $W^U$ and $W^V$ be independent two-sided standard Brownian motions, independent of $\widetilde{\mathcal E}$. Define
\begin{align}
\widetilde{\mathcal U}(t)&=\delta\int_{-\infty}^te^{-a(t-s)}\widetilde{\mathcal E}(s)\,ds,\notag\\
\widetilde{\mathcal Z}^{U}(t)&=\int_{-\infty}^te^{-(a-1/2)(t-s)}\sigma\bigl(\widetilde{\mathcal U}(s),\widetilde{\mathcal E}(s)\bigr)\,dW^U(s),\label{ecua_stationary-ZU-two-sided}\\
\widetilde{\mathcal Z}^{V}(t)&=\int_{-\infty}^te^{-(b-1/2)(t-s)}\sigma\bigl(\widetilde{\mathcal U}(s),\widetilde{\mathcal E}(s)\bigr)\,dW^V(s).\label{ecua_stationary-ZV-two-sided}
\end{align}
The stochastic integrals are well defined because $a,b>1/2$ and $\sigma$ is bounded. Write $(\mathcal U,\mathcal E,\mathcal Z^U,\mathcal Z^V)$ for the restriction of this stationary process to $[0,\infty)$.

For later reference, the generator $\mathcal G$ of the stationary limiting process acts on functions $f$ such that
$f(\cdot,e,\cdot,\cdot)\in C_c^2(\Rb^3)$ for each $e\in\{-1,+1\}$ as follows:
\begin{align*}
\mathcal Gf(u,e,z_1,z_2)&=(-au+\delta e)\partial_u f-\left(a-\frac12\right)z_1\partial_{z_1}f-\left(b-\frac12\right)z_2\partial_{z_2}f+
\frac{\sigma^2(u,e)}{2}
\left(\partial_{z_1z_1}f+\partial_{z_2z_2}f
\right)\\
&+c\,[f(u,-e,z_1,z_2)-f(u,e,z_1,z_2)].
\end{align*}

\begin{proof}[Proof of Theorem \ref{teo_hybrid-fluctuation-limit}]
Write $\theta_U=a$ and $\theta_V=b$, and set $\lambda_\ell=\theta_\ell-1/2>0$ for $\ell\in\{U,V\}$. We proceed in five steps.

\emph{Step 1: Product and moment bounds.} For $\theta>0$, set
$$P_{k,n}^{(\theta)}=\prod_{j=k}^{n-1}\left(1-\frac{\theta}{j+\nu+1}\right).$$
Choose $N_\theta$ sufficiently large that all factors with $j\geq N_\theta$ are positive. The gamma-ratio argument used for \eqref{ecua_Pkn-upper-bound}, with the finitely many preceding factors absorbed into the constant, gives
\begin{equation}\label{ecua_general-product-bound}
\left|P_{k,n}^{(\theta)}\right|\leq C_\theta\left(\frac{k+\nu+1}{n+\nu+1}\right)^\theta,\qquad 0\leq k\leq n.
\end{equation}
If one of the initial factors is zero, the corresponding product vanishes and the same estimate is immediate.

Iteration of the centered recursions and orthogonality of the martingale differences yield, for $\ell\in\{U,V\}$,
\begin{equation}\label{ecua_centered-second-moment-bound}
\Eb[(Y_n^\ell)^2]\leq C(n+\nu+1)^{-2\theta_\ell}\sum_{k=0}^{n-1}(k+\nu+1)^{2\theta_\ell-2}\leq\frac{C}{n}.
\end{equation}
The product estimate \eqref{ecua_general-product-bound} also gives $\bar V_n=O(n^{-b})$ without requiring positivity of the finitely many initial factors.

Put $z_m^\ell=\sqrt m\,Y_m^\ell$ for $m\geq1$. Expanding \eqref{ecua_centered-U-recursion}--\eqref{ecua_centered-V-recursion} gives
\begin{equation}\label{ecua_scaled-one-step}
z_{m+1}^\ell-z_m^\ell=\frac{1/2-\theta_\ell}{m}z_m^\ell+\frac{\Delta M_{m+1}^\ell}{\sqrt m}+\varepsilon_{m+1}^\ell,
\end{equation}
where
$$|\varepsilon_{m+1}^\ell|\leq C\left(m^{-2}|z_m^\ell|+m^{-3/2}\right).$$
By \eqref{ecua_centered-second-moment-bound},
$$\sup_{m\geq1}\Eb[(z_m^\ell)^2]<\infty.$$
Consequently, for each $T>0$,
\begin{equation}\label{ecua_accumulated-remainder}
\Eb\left[\sum_{m=n}^{\lfloor ne^T\rfloor}|\varepsilon_{m+1}^\ell|\right]\leq C_Tn^{-1/2}\longrightarrow0.
\end{equation}

\emph{Step 2: Joint martingale limit on a finite logarithmic window.}
Define
$$N_n^\ell(t)=\sum_{m=n}^{m_n(t)-1}\frac{\Delta M_{m+1}^\ell}{\sqrt m}, \qquad \ell\in\{U,V\}.$$
Their predictable quadratic and cross-variations, with $Q_m$ and $\Xi_m$ as in \eqref{ecua_Qn-Xin}, are
\begin{align*}
\langle N_n^U\rangle(t)=\langle N_n^V\rangle(t)&=\sum_{m=n}^{m_n(t)-1}\frac{Q_m}{m},\\
\langle N_n^U,N_n^V\rangle(t)&=\sum_{m=n}^{m_n(t)-1}\frac{\Xi_m}{m}.
\end{align*}

Since
$$Q_m=\sigma^2(U_m,E_m)-\frac{d_m^2}{2},\qquad d_m=(\rho-\kappa)V_m,$$
the Lipschitz continuity of $\sigma^2(\,\cdot\,,e)$ on the relevant compact set gives
$$\left|Q_m-\sigma^2(\bar U_m,E_m)\right|\leq C\left(|Y_m^U|+V_m^2\right).$$
Moreover,
$$|\Xi_m|=\frac12|d_m(1-2p_m)|\leq C|V_m|.$$

By Lemma \ref{lema_martingale-negligible} and Proposition \ref{propo_synchronization},
$$\sup_{m\geq n}\left(\left|Q_m-\sigma^2(\bar U_m,E_m)\right|+|\Xi_m|\right)\longrightarrow0\qquad\text{almost surely}.$$

Since $Y_m^U=U_m-\bar U_m=R_m$, Lemma \ref{lema_martingale-negligible}
implies that, for every $T>0$,
$$\sup_{0\le t\le T}\left|\mathcal U_n(t)-\bar{\mathcal U}_n(t)\right|\le\sup_{m\ge n}|R_m|\longrightarrow 0\qquad\text{almost surely}.$$
Consequently, Theorem \ref{teo_critical-functional-convergence}
and Slutsky's theorem yield $(\bar{\mathcal U}_n,\mathcal E_n)\Longrightarrow(\mathcal U,\mathcal E)$ on every compact logarithmic time interval.

For $\ell\in\{U,V\}$, it follows that
$$\sup_{0\leq t\leq T}\left|\langle N_n^\ell\rangle(t)-\sum_{m=n}^{m_n(t)-1}\frac{\sigma^2(\bar U_m,E_m)}{m}\right|\leq C_T\sup_{m\geq n}\left|Q_m-\sigma^2(\bar U_m,E_m)\right|\longrightarrow0 \qquad\text{almost surely}.$$
Moreover, since the logarithmic-time interpolations are constant on the intervals
$$\left[\log\left(\frac mn\right),\log\left(\frac{m+1}{n}\right)\right)$$
and $\log((m+1)/m)=m^{-1}+O(m^{-2})$, the boundedness of $\sigma^2$ gives
$$\sup_{0\leq t\leq T}\left|\sum_{m=n}^{m_n(t)-1}\frac{\sigma^2(\bar U_m,E_m)}{m}
-\int_0^t\sigma^2\bigl(\bar{\mathcal U}_n(s),\mathcal E_n(s)\bigr)\,ds\right|=O(n^{-1}).$$
Since $\sigma^2$ is bounded and continuous, the continuous mapping theorem and Slutsky's theorem now yield
\begin{equation}\label{ecua_bracket-limit}
\left(\bar{\mathcal U}_n,\mathcal E_n,\langle N_n^U\rangle,\langle N_n^V\rangle\right)\Longrightarrow\left(\mathcal U,\mathcal E,\mathcal Q,\mathcal Q\right),
\qquad \mathcal Q(t)=\int_0^t\sigma^2(\mathcal U(s),\mathcal E(s))\,ds.
\end{equation}
The cross-variation satisfies
$$\sup_{t\leq T}\left|\langle N_n^U,N_n^V\rangle(t)\right|\leq C_T\sup_{m\geq n}|V_m|\longrightarrow0 \qquad\text{almost surely}.$$
Equivalently, define the matrix-valued predictable quadratic variation by
$$\langle\mathbf N_n\rangle(t)=\begin{pmatrix}\langle N_n^U\rangle(t)&\langle N_n^U,N_n^V\rangle(t)\\
\langle N_n^U,N_n^V\rangle(t)&\langle N_n^V\rangle(t)\end{pmatrix},
\qquad \mathbf N_n=(N_n^U,N_n^V)^\top.$$
Combining \eqref{ecua_bracket-limit} with the preceding estimate gives
$\left(\bar{\mathcal U}_n,\mathcal E_n,\langle\mathbf N_n\rangle\right)
\Longrightarrow \left(\mathcal U,\mathcal E,\mathcal Q I_2\right)$. Moreover, since $\lvert\Delta M_{m+1}^\ell\rvert\leq1$,
$$\sup_{0\leq t\leq T}\left|\Delta\mathbf N_n(t)\right|\leq\frac{\sqrt{2}}{\sqrt n}\longrightarrow0.$$

Hence, the conditional Lindeberg condition holds. Since the limiting predictable quadratic variation is continuous, the martingale tightness criterion \cite[Theorem VI.4.13 and Proposition VI.3.26]{Jacod_Shiryaev_03} gives $C$-tightness of $(N_n^U,N_n^V)$. In particular, every subsequential martingale limit is continuous and has predictable quadratic variation $\mathcal Q I_2$.

To identify the Gaussian noises jointly with the environmental process, we use a joint martingale problem. Let $f(u,e,z_1,z_2)$ be bounded and suppose that $f(\,\cdot\,,e,\,\cdot\,,\,\cdot\,)\in C_b^3(\Rb^3)$ for each $e\in\{-1,+1\}$. For $m\geq n$, define
$$N_{n,m}^\ell=\sum_{k=n}^{m-1}\frac{\Delta M_{k+1}^\ell}{\sqrt k},\qquad\mathbf Y_{m,n}=(\bar U_m,E_m,N_{n,m}^U,N_{n,m}^V).$$
Also set
$$\zeta_{m+1}^\ell=\frac{\Delta M_{m+1}^\ell}{\sqrt m},\qquad \ell\in\{U,V\},$$
and $\Delta u_m=r_m(-a\bar U_m+\delta E_m)$. The conditional moments of the scaled martingale increments are
\begin{align*}
\Eb[\zeta_{m+1}^U\mid\F_m]&=\Eb[\zeta_{m+1}^V\mid\F_m]=0,\qquad
\Eb[(\zeta_{m+1}^U)^2\mid\F_m]=\Eb[(\zeta_{m+1}^V)^2\mid\F_m]=\frac{Q_m}{m},\qquad\text{and}\\
\Eb[\zeta_{m+1}^U\zeta_{m+1}^V\mid\F_m]&=\frac{\Xi_m}{m}.
\end{align*}

Conditionally on $\F_m$, the switch variable $B_{m+1}$ is independent of $(\zeta_{m+1}^U,\zeta_{m+1}^V)$. Conditioning first on $B_{m+1}$, keeping the possible sign change unexpanded, and applying a third-order Taylor expansion in the continuous variables gives
\begin{align*}
\Eb\left[f(\mathbf Y_{m+1,n})-f(\mathbf Y_{m,n})\mid\F_m\right]&=\Delta u_m\,\partial_u f(\mathbf Y_{m,n})+q_{m+1}\left[f(\bar U_m,-E_m,N_{n,m}^U,N_{n,m}^V)-f(\mathbf Y_{m,n})\right]\\
&+\frac{Q_m}{2m}\left(\partial_{z_1z_1}f(\mathbf Y_{m,n})+\partial_{z_2z_2}f(\mathbf Y_{m,n})\right)+\frac{\Xi_m}{m}
\partial_{z_1z_2}f(\mathbf Y_{m,n})+R_{m,n},
\end{align*}
where, uniformly for $n\leq m\leq ne^T$,
$$|R_{m,n}|\leq C_f\left(m^{-3/2}+m^{-2}\right).$$
The two terms account, respectively, for the third-order martingale remainder and the remaining continuous-variable increments, including those multiplied by $q_{m+1}=O(m^{-1})$.

Since
$$r_m=\frac1m+O(m^{-2})$$
and
$$q_{m+1}=\frac{c}{m}+\frac{mq_{m+1}-c}{m},$$
the preceding expansion can be written as
$$\Eb\left[f(\mathbf Y_{m+1,n})-f(\mathbf Y_{m,n})\mid\F_m\right]=\frac1m\mathcal G_0f(\mathbf Y_{m,n})+\eta_{m,n},$$
where
\begin{align*}
\mathcal G_0f(u,e,z_1,z_2)=(-au+\delta e)\partial_u f+c\left[
f(u,-e,z_1,z_2)-f(u,e,z_1,z_2)\right]+\frac{\sigma^2(u,e)}{2}\left(\partial_{z_1z_1}f+\partial_{z_2z_2}f\right),
\end{align*}
and, uniformly for $n\leq m\leq ne^T$,
\begin{align*}
|\eta_{m,n}|\leq C_f\left(m^{-3/2}+m^{-2}\right) +\frac{C_f}{m}
\left(|mq_{m+1}-c|+\left|Q_m-\sigma^2(\bar U_m,E_m)\right|+|\Xi_m|\right).
\end{align*}

Since $mq_{m+1}\to c$, the preceding estimates, together with
$$\sum_{m=n}^{\lfloor ne^T\rfloor}\left(m^{-3/2}+m^{-2}\right)=O(n^{-1/2}),\qquad \sum_{m=n}^{\lfloor ne^T\rfloor}\frac1m\leq C_T,$$
imply that
$$\sup_{0\leq t\leq T}\left|\sum_{m=n}^{m_n(t)-1}\eta_{m,n}\right|\longrightarrow0\qquad\text{in probability}.$$

Using the same logarithmic Riemann-sum estimate as above, define
$$A_n^f(t)=f\bigl(\bar U_{m_n(t)},E_{m_n(t)},N_n^U(t),N_n^V(t)\bigr)-f(\bar U_n,E_n,0,0)-\int_0^t\mathcal G_0f\bigl(\bar{\mathcal U}_n(s),\mathcal E_n(s),N_n^U(s),N_n^V(s)\bigr)\,ds.$$
The preceding conditional expansion gives $A_n^f=M_n^f+r_n^f$, where $M_n^f$ is a martingale and
$$\sup_{0\leq t\leq T}|r_n^f(t)|\longrightarrow0\qquad\text{in probability}.$$

The convergence of $(\bar{\mathcal U}_n,\mathcal E_n)$ established above and the preceding $C$-tightness result give joint tightness of $(\bar{\mathcal U}_n,\mathcal E_n,N_n^U,N_n^V)$. Since $f$, its derivatives, and the coefficients of $\mathcal G_0$ are bounded, the martingales $M_n^f$ are uniformly integrable on compact time intervals. Thus, the martingale property passes to every subsequential limit, which therefore solves the $\mathcal G_0$-martingale problem with initial law $(U^{(c)},\mathcal E(0),0,0)$.

This martingale problem is well posed because it is triangular. The environmental coordinate is the rate-$c$ telegraph process, and, given its path and the initial state, the first coordinate is its unique linear filter. Given the resulting background $(\mathcal U,\mathcal E)$, the last two coordinates are uniquely represented by stochastic integrals of $\sigma(\mathcal U,\mathcal E)$ against two independent Brownian motions, independent of the telegraph process. Consequently,
\begin{equation}\label{ecua_joint-martingale-environment-limit}
(\bar{\mathcal U}_n,\mathcal E_n,N_n^U,N_n^V)\Longrightarrow\left(\mathcal U,\mathcal E,\int_0^\cdot\sigma(\mathcal U(s),\mathcal E(s))\,dW^U(s),\int_0^\cdot\sigma(\mathcal U(s),\mathcal E(s))\,dW^V(s)\right).
\end{equation}

\emph{Step 3: Tightness and finite-window dynamics.}
Summing \eqref{ecua_scaled-one-step}, using the logarithmic Riemann-sum estimate and applying \eqref{ecua_accumulated-remainder}, we obtain
\begin{equation}\label{ecua_scaled-integral-equation}
\mathcal Z_n^\ell(t)=\mathcal Z_n^\ell(0)-\lambda_\ell\int_0^t\mathcal Z_n^\ell(s)\,ds+N_n^\ell(t)+\widetilde\varepsilon_n^\ell(t),
\end{equation}
where, for every $T>0$,
$$\sup_{0\leq t\leq T}|\widetilde\varepsilon_n^\ell(t)|\longrightarrow0 \qquad\text{in probability}.$$

By \eqref{ecua_centered-second-moment-bound}, the initial vector $\bigl(\mathcal Z_n^U(0),\mathcal Z_n^V(0)\bigr)$ is tight. Step 2 gives joint tightness of $(\bar{\mathcal U}_n,\mathcal E_n,N_n^U,N_n^V)$, while $(N_n^U,N_n^V)$ is $C$-tight. Since the solution map associated with \eqref{ecua_scaled-integral-equation} is continuous at continuous driving paths, the vanishing remainder and the tightness of the initial vector imply joint tightness of $(\bar{\mathcal U}_n,\mathcal E_n,\mathcal Z_n^U,\mathcal Z_n^V)$.

Along any convergent subsequence, passing to the limit in \eqref{ecua_scaled-integral-equation} and using \eqref{ecua_joint-martingale-environment-limit} shows that the limiting coordinates satisfy \eqref{ecua_limit-SDE-ZU}--\eqref{ecua_limit-SDE-ZV}, with an initially unspecified joint entrance law.

\emph{Step 4: Stationary entrance law.}
Fix $S>0$ and let $j_n(S)=\lfloor ne^{-S}\rfloor$. For this step,
extend the logarithmic-time interpolations to $[-S,T]$ using the same
definitions as before. For $\ell\in\{U,V\}$ and $-S\leq t\leq T$,
define the truncated scaled residual
$$\mathcal Z_{n,S}^{\ell}(t)=\sqrt{m_n(t)}\sum_{k=j_n(S)}^{m_n(t)-1}P_{k+1,m_n(t)}^{(\theta_\ell)}r_k\Delta M_{k+1}^{\ell}.$$
For $m\geq j_n(S)$, iteration of the centered recursion from $j_n(S)$ gives
$$Y_m^\ell=P_{j_n(S),m}^{(\theta_\ell)}Y_{j_n(S)}^\ell+\sum_{k=j_n(S)}^{m-1}P_{k+1,m}^{(\theta_\ell)}r_k\Delta M_{k+1}^\ell.$$
Consequently, for $0\leq t\leq T$,
$$\mathcal Z_n^\ell(t)=\sqrt{m_n(t)}P_{j_n(S),m_n(t)}^{(\theta_\ell)}Y_{j_n(S)}^\ell+\mathcal Z_{n,S}^\ell(t).$$

Since $\lambda_\ell>0$, the product and moment bounds give
\begin{align}\label{ecua_stationary-entrance-discrete-tail}
&\Eb\left[\sup_{0\leq t\leq T}m_n(t)\left(P_{j_n(S),m_n(t)}^{(\theta_\ell)}Y_{j_n(S)}^\ell\right)^2\right] \notag\\
&\qquad\leq C\sup_{m\geq n}m\left(\frac{j_n(S)}{m}\right)^{2\theta_\ell}\frac1{j_n(S)}\leq C\left(\frac{j_n(S)}{n}\right)^{2\theta_\ell-1}\leq Ce^{-2\lambda_\ell S}.
\end{align}
Thus, the contribution from times preceding logarithmic time $-S$ becomes negligible uniformly on $[0,T]$ as $n\to\infty$ followed by $S\to\infty$.

On the limiting probability space, define
$$\widetilde{\mathcal Z}_S^\ell(t)=\int_{-S}^{t}e^{-\lambda_\ell(t-s)}\sigma\bigl(\widetilde{\mathcal U}(s),\widetilde{\mathcal E}(s)\bigr)\,dW^\ell(s), \qquad -S\leq t\leq T.$$

For fixed $S>0$, set
$$N_{n,S}^{\ell}(t)=\sum_{k=j_n(S)}^{m_n(t)-1}\frac{\Delta M_{k+1}^{\ell}}{\sqrt{k}},\qquad \ell\in\{U,V\}, \qquad -S\leq t\leq T.$$
Since $j_n(S)\sim ne^{-S}\to\infty$, the arguments of Step 2 apply with $j_n(S)$ as the initial index. In particular, the diagonal predictable quadratic variations converge jointly to
$$\mathcal Q_S(t)=\int_{-S}^{t}\sigma^2\bigl(\widetilde{\mathcal U}(s),\widetilde{\mathcal E}(s)\bigr)\,ds,$$
the cross-variation converges uniformly to zero, and the maximal
jump is bounded by $j_n(S)^{-1/2}$. Therefore,
$\left(\bar{\mathcal U}_n,\mathcal E_n,N_{n,S}^{U},N_{n,S}^{V}\right)\Longrightarrow\left(\widetilde{\mathcal U},\widetilde{\mathcal E}, M_S^{U}, M_S^{V}\right)$ in $D([-S,T],\Rb\times\{-1,+1\}\times\Rb^2)$, where
$$M_S^\ell(t)=\int_{-S}^{t}\sigma\bigl(\widetilde{\mathcal U}(s),\widetilde{\mathcal E}(s)\bigr)\,dW^\ell(s).$$

The scaled one-step expansion used in Step 3, now with zero initial condition at $j_n(S)$, gives
$$\mathcal Z_{n,S}^{\ell}(t)=-\lambda_\ell\int_{-S}^{t}\mathcal Z_{n,S}^{\ell}(s)\,ds+N_{n,S}^{\ell}(t)+\varepsilon_{n,S}^{\ell}(t),$$
where, for fixed $S$ and $T$,
$$\sup_{-S\leq t\leq T}\left|\varepsilon_{n,S}^{\ell}(t)\right|\longrightarrow0\qquad\text{in probability}.$$
The solution map associated with this linear integral equation is continuous at continuous driving paths. Hence, the continuous mapping theorem and the variation-of-constants formula yield $\left(\bar{\mathcal U}_n,\mathcal E_n,\mathcal Z_{n,S}^{U},\mathcal Z_{n,S}^{V}\right)\Longrightarrow\left(\widetilde{\mathcal U},\widetilde{\mathcal E},\widetilde{\mathcal Z}_S^{U},\widetilde{\mathcal Z}_S^{V}\right)$ in $D([-S,T],\Rb\times\{-1,+1\}\times\Rb^2)$.

The difference between the full stationary convolution and its truncation at time $-S$ is
$$\widetilde{\mathcal Z}^\ell(t)-\widetilde{\mathcal Z}_S^\ell(t)=\int_{-\infty}^{-S}e^{-\lambda_\ell(t-s)}\sigma\bigl(\widetilde{\mathcal U}(s),\widetilde{\mathcal E}(s)\bigr)\,dW^\ell(s).$$
Since, for $t\geq0$,
$$\int_{-\infty}^{-S}e^{-\lambda_\ell(t-s)}\sigma\bigl(\widetilde{\mathcal U}(s),\widetilde{\mathcal E}(s)\bigr)\,dW^\ell(s)=e^{-\lambda_\ell t}
\int_{-\infty}^{-S}e^{\lambda_\ell s}\sigma\bigl(\widetilde{\mathcal U}(s),\widetilde{\mathcal E}(s)\bigr)\,dW^\ell(s),$$
It\^o's isometry and the boundedness of $\sigma$ give \begin{equation}\label{ecua_stationary-entrance-limit-tail}
\Eb\left[\sup_{0\leq t\leq T}\left|\int_{-\infty}^{-S}e^{-\lambda_\ell(t-s)}\sigma\bigl(\widetilde{\mathcal U}(s),\widetilde{\mathcal E}(s)
\bigr)\,dW^\ell(s)\right|^2\right]\leq Ce^{-2\lambda_\ell S}.
\end{equation}

By Markov's inequality, \eqref{ecua_stationary-entrance-discrete-tail} implies that, for every $\varepsilon>0$,
$$\lim_{S\to\infty}\limsup_{n\to\infty}\Pb\left(\max_{\ell\in\{U,V\}}\sup_{0\leq t\leq T}\left|\mathcal Z_n^\ell(t)-\mathcal Z_{n,S}^\ell(t)\right|>\varepsilon\right)=0.$$
Similarly, \eqref{ecua_stationary-entrance-limit-tail} gives
$$\lim_{S\to\infty}\Pb\left(\max_{\ell\in\{U,V\}}\sup_{0\leq t\leq T}\left|\widetilde{\mathcal Z}^\ell(t)-\widetilde{\mathcal Z}_S^\ell(t)\right|>\varepsilon\right)=0.$$
Therefore, the converging-together theorem, first as $n\to\infty$ for fixed $S$ and then as $S\to\infty$, gives the asserted joint functional convergence to the restrictions of the two-sided representations \eqref{ecua_stationary-ZU-two-sided} and \eqref{ecua_stationary-ZV-two-sided}. The displayed generator follows from the limiting drift and diffusion characteristics together with the rate-$c$ sign changes.

\emph{Step 5: Uniqueness of the stationary law.}
Existence follows from the two-sided representations \eqref{ecua_stationary-ZU-two-sided}--\eqref{ecua_stationary-ZV-two-sided}. Let $\Pi_1$ and $\Pi_2$ be two stationary laws of the full system. By Proposition \ref{propo_critical-invariant-distribution}, their $(\mathcal U,\mathcal E)$-marginals coincide. Disintegrating over this common marginal, we may couple two corresponding stationary solutions with the same background process and the same Brownian drivers. If
$D^\ell(t)=\mathcal Z^{\ell,1}(t)-\mathcal Z^{\ell,2}(t),\ell\in\{U,V\},$ then the common diffusion term cancels and
$$D^\ell(t)=e^{-\lambda_\ell t}D^\ell(0)\longrightarrow0\qquad\text{almost surely}.$$
Put $\mathbf X^i(t)=(\mathcal U(t),\mathcal E(t),\mathcal Z^{U,i}(t),\mathcal Z^{V,i}(t))$. For every bounded Lipschitz function $g$ on the full state space, let $\operatorname{Lip}(g)$ denote its Lipschitz constant with respect to the product metric. Then stationarity and the coupling give
$$\left|\int g\,d\Pi_1-\int g\,d\Pi_2\right|=\left|\Eb[g(\mathbf X^1(t))-g(\mathbf X^2(t))]\right|\leq\Eb\left[\min\left\{2\|g\|_\infty,\operatorname{Lip}(g)\sum_{\ell\in\{U,V\}}|D^\ell(t)|\right\}\right]\longrightarrow0.$$
Thus $\Pi_1=\Pi_2$, since bounded Lipschitz functions determine probability laws.
\end{proof}

\subsection{Stationary moments and interaction-induced correlation}
Let $\overline\sigma^2=\Eb\left[\sigma^2(\mathcal U(0),\mathcal E(0))\right]$.

\begin{proposition}\label{propo_stationary-fluctuation-moments}
The averaged variance coefficient is
\begin{equation}\label{ecua_average-volatility-explicit}
\overline\sigma^2=\frac18-\frac{\delta^2(1+2ac)}{2a(a+2c)}.
\end{equation}
Moreover,
\begin{align*}
\Eb[(\mathcal Z^U(0))^2]&=\frac{\overline\sigma^2}{2a-1},\qquad
\Eb[(\mathcal Z^V(0))^2]=\frac{\overline\sigma^2}{2b-1},\qquad\text{and}\qquad
\Eb[\mathcal Z^U(0)\mathcal Z^V(0)]=0.\end{align*}
\end{proposition}

\begin{proof}
The invariant identities in \eqref{ecua_critical-moments} give
$$\Eb[\mathcal U(0)^2]=\frac{\delta^2}{a(a+2c)}\qquad\text{and}\qquad \Eb[\mathcal U(0)\mathcal E(0)]=\frac{\delta}{a+2c}.$$
Since
$$p(\mathcal U(0),\mathcal E(0))-\frac12=(1-a)\mathcal U(0)+\delta\mathcal E(0),$$
we obtain
\begin{align*}
\Eb\left[\left(p(\mathcal U(0),\mathcal E(0))-\frac12\right)^2\right]
&=(1-a)^2\Eb[\mathcal U(0)^2]+2\delta(1-a)\Eb[\mathcal U(0)\mathcal E(0)]+\delta^2\\
&=\frac{\delta^2(1+2ac)}{a(a+2c)}.
\end{align*}
Combining this identity with
$$\overline\sigma^2=\frac18-\frac12\Eb\left[\left(p(\mathcal U(0),\mathcal E(0))-\frac12\right)^2 \right]$$
proves \eqref{ecua_average-volatility-explicit}.

Applying It\^o's formula to the stationary equations \eqref{ecua_limit-SDE-ZU}--\eqref{ecua_limit-SDE-ZV} and then using stationarity yields
$$(2a-1)\Eb[(\mathcal Z^U(0))^2]=\overline\sigma^2 \qquad\text{and}\qquad (2b-1)\Eb[(\mathcal Z^V(0))^2]=\overline\sigma^2.$$

Conditionally on the complete path of $\widetilde{\mathcal E}$, the process $\widetilde{\mathcal U}$, and hence the integrands in \eqref{ecua_stationary-ZU-two-sided}--\eqref{ecua_stationary-ZV-two-sided}, are fixed. Since $W^U$ and $W^V$ are independent and are also independent of $\widetilde{\mathcal E}$, the two stochastic integrals are conditionally independent and centered. Their conditional product expectation is therefore zero, and the tower property gives $\Eb[\mathcal Z^U(0)\mathcal Z^V(0)]=0$. Unconditional independence does not follow, however, because both conditional variances depend on the same environmental path.
\end{proof}

\begin{corollary}
\label{corol_component-fluctuation-correlation}
For $i=1,2$, define
$$\mathcal R_n^i(t)=\sqrt{m_n(t)}\left(X_{m_n(t)}^i-\frac12-\bar U_{m_n(t)}\right).$$
Then $(\mathcal R_n^1,\mathcal R_n^2)
\Longrightarrow
(\mathcal Z^U+\mathcal Z^V,\mathcal Z^U-\mathcal Z^V)$ in $D([0,T],\Rb^2)$. At stationarity, both limiting components have variance
$$\overline\sigma^2\left(\frac{1}{2a-1}+\frac{1}{2b-1}\right),$$
and their contemporaneous correlation is
\begin{equation}\label{ecua_interaction-correlation}
\operatorname{Corr}\left(\mathcal Z^U(0)+\mathcal Z^V(0),\mathcal Z^U(0)-\mathcal Z^V(0)\right)=\frac{b-a}{a+b-1}=\frac{2\kappa}{2a+2\kappa-1}=\frac{2\kappa}{1-2\rho}.
\end{equation}
\end{corollary}

\begin{proof}
By \eqref{ecua_inverse-UV}, for every $t\in[0,T]$,
$$\mathcal R_n^1(t)=\mathcal Z_n^U(t)+\mathcal Z_n^V(t)+\sqrt{m_n(t)}\,\bar V_{m_n(t)}
\qquad\text{and}\qquad
\mathcal R_n^2(t)=\mathcal Z_n^U(t)-\mathcal Z_n^V(t)-\sqrt{m_n(t)}\,\bar V_{m_n(t)}.$$ Since $b>1/2$ and $\bar V_m=O(m^{-b})$, the last terms vanish uniformly on compact logarithmic-time intervals. The functional convergence follows from Theorem \ref{teo_hybrid-fluctuation-limit} and the continuous mapping theorem.

Proposition \ref{propo_stationary-fluctuation-moments} gives
\begin{align*}
\operatorname{Var}(\mathcal Z^U(0)\pm\mathcal Z^V(0))&=\overline\sigma^2\left(\frac{1}{2a-1}+\frac{1}{2b-1}\right),\\
\operatorname{Cov}(\mathcal Z^U(0)+\mathcal Z^V(0),\mathcal Z^U(0)-\mathcal Z^V(0))&=\overline\sigma^2\left(\frac{1}{2a-1}-\frac{1}{2b-1}\right).
\end{align*}
Dividing the covariance by the common variance and using $b=a+2\kappa$ gives \eqref{ecua_interaction-correlation}. Since $a>1/2$ and $\kappa\geq0$, it follows that $0\leq\chi<1$. If $a$ is known and $\chi$ denotes the correlation in \eqref{ecua_interaction-correlation}, direct inversion gives
$$\kappa=\frac{\chi(2a-1)}{2(1-\chi)},\qquad\text{and}\qquad \rho= 1-a-\kappa.$$
Thus, at the level of the limiting covariance structure, the pair $(a,\chi)$ determines $(\rho,\kappa)$ within the admissible parameter region.
\end{proof}

\section{Concluding remarks}\label{seccion_conclusion}

We have studied the fluctuation behavior of two interacting reinforced occupation processes when a common Markov environment evolves on the same asymptotic scale as the stochastic-approximation recursion. At the critical switching scale, the environmental response survives in the first-order limit as a nondegenerate telegraph-driven PDMP. Consequently, centering at a deterministic equilibrium does not isolate the intrinsic
square-root fluctuations.

The main contribution is an exact discrete random centering that removes the accumulated environmental response before rescaling. Under $a>1/2$, the centered collective and synchronization coordinates converge
jointly with the first-order background to a stationary hybrid PDMP--diffusion system. Its two fluctuation modes have contraction rates $a-1/2$ and $b-1/2$, respectively, and share an environment-dependent volatility. The two-sided stochastic-convolution representation specifies the stationary entrance law inherited from the complete prelimit history.

The first-order analysis provides the averaging--critical--freezing classification associated with the competition between the stochastic-approximation gain and the environmental switching probability. The synchronization coordinate vanishes in all three regimes, whereas the collective coordinate averages to zero for $0<\gamma<1$, remains random on the logarithmic time scale for $\gamma=1$, and converges to an environment-dependent equilibrium for $\gamma>1$. In the critical regime, the invariant marginal distribution
of the first-order PDMP is an affine transformation of a symmetric Beta distribution.

Conditionally on the complete environmental path, the collective and synchronization fluctuation modes are independent centered Gaussian processes. Their recombination yields the correlation in (47), which is positive precisely when $\kappa>0$. Since the first-order dynamics depend on $(\rho,\kappa)$ only through $a=1-\rho-\kappa$, while the second-order correlation also contains the difference between the two
contraction rates, the pair $(a,\chi)$ provides an asymptotic identifiability relation for the self-interaction and cross-interaction coefficients.

The value $1/2$ is the natural boundary of the square-root Gaussian regime for each contraction mode. Further questions include the fluctuation behavior at this boundary, the corresponding nonstandard scalings in the subcritical contraction regime, and extensions of the random-centering principle to larger interaction networks and finite-state switching environments.

\appendix
\section{Proof of the invariant-law formula}\label{appendice_invariant-law}

\begin{proof}[Proof of Proposition \ref{propo_critical-invariant-distribution}] Existence follows from the two-sided stationary construction. Let $\pi$ denote the resulting invariant probability measure and set $\pi_e(du)=\pi(du,\{e\})$ for $e\in\{-1,+1\}$. The $\mathcal U$-marginal has no atoms at $\pm\delta/a$: attaining an endpoint would require the environmental sign to remain constant throughout the entire past, an event of probability zero.

Writing $D$ for the distributional derivative, stationarity yields
$$D\bigl((\delta-au)\pi_+\bigr)=c(\pi_- - \pi_+)\qquad\text{and}\qquad D\bigl((-\delta-au)\pi_-\bigr)=c(\pi_+ - \pi_-).$$
Consequently, $(\delta-au)\pi_+$ and $(-\delta-au)\pi_-$ have locally bounded-variation densities. Since the corresponding velocities do not
vanish in the interior of $I=(-\delta/a, \delta/a)$, both $\pi_+$ and $\pi_-$ are absolutely continuous on $I$. Denote their densities by $f_+$ and $f_-$. They satisfy
$$-\frac{d}{du}[(-au+\delta)f_+]-cf_++cf_-=0 \qquad\text{and}\qquad -\frac{d}{du}[(-au-\delta)f_-]+cf_+-cf_-=0.$$

Adding these equations shows that the stationary flux $J(u)=(\delta-au)f_+(u)+(-\delta-au)f_-(u)$ is constant on $I$. This constant must be zero. Indeed, if $J>0$, then
$$(\delta-au)f_+(u)=J+(\delta+au)f_-(u)\geq J,$$
so $f_+(u)\geq J/(\delta-au)$, which is not integrable near $\delta/a$. Similarly, if $J<0$, then
$$(\delta+au)f_-(u)=-J+(\delta-au)f_+(u)\geq -J,$$
which would make $f_-$ nonintegrable near $-\delta/a$. Hence, $J=0$.

Writing $f_c=f_++f_-$, the zero-flux identity gives
$$f_+-f_-=\frac{au}{\delta}f_c.$$
Therefore,
$$f_+(u)=\frac12\left(1+\frac{au}{\delta}\right)f_c(u),\qquad f_-(u)=\frac12\left(1-\frac{au}{\delta}\right)f_c(u),$$
which proves \eqref{ecua_joint-densities}. Substitution into either stationary balance equation gives
$$\frac{f_c'(u)}{f_c(u)}=-\frac{2a(c-a)u}{\delta^2-a^2u^2}.$$
Integration and normalization yield \eqref{ecua_critical-density-derived}. The affine change of variables
$$y=\frac12\left(1+\frac{au}{\delta}\right)$$
then gives \eqref{ecua_critical-beta-derived}.

It remains to prove uniqueness. Couple two copies of the environmental process independently until their signs first agree, and use the same Poisson switching clock thereafter. Their coupling time $\tau$ is finite almost surely. Once the environmental signs agree, the corresponding continuous coordinates satisfy
$$|\mathcal U_1(t)-\mathcal U_2(t)|=e^{-a(t-\tau)}|\mathcal U_1(\tau)-\mathcal U_2(\tau)|,\qquad t\geq\tau.$$
Thus, the distance between the coupled processes converges to zero almost surely. Applying this coupling to two invariant initial distributions and using the bounded convergence theorem for bounded Lipschitz functions shows that the two invariant distributions coincide.
\end{proof}

\bibliographystyle{unsrt}  
\bibliography{references}

\end{document}